\documentclass{elsarticle}%
\usepackage{amsfonts}
\usepackage{amsthm}
\usepackage{amsmath}
\usepackage{amssymb}
\usepackage{bbold}
\usepackage{mathtools}
\usepackage{xcolor}
\usepackage{soul}
\usepackage{mathrsfs}
\usepackage{tabularx}
\usepackage[top=1in, bottom=1in, left=1in, right=1in]{geometry}
\usepackage{graphicx}%
\providecommand{\U}[1]{\protect\rule{.1in}{.1in}}

\newtheorem{theorem}{Theorem}[section]
\newtheorem{lemma}{Lemma}[section]

\theoremstyle{definition}
\newtheorem{definition}{Definition}[section]
\theoremstyle{remark}

\DeclareMathOperator*{\esssup}{ess\,sup}
\numberwithin{equation}{section}
\begin{document}
	\begin{frontmatter}
		
		\title{Pontryagin maximum principle for fractional delay differential equations and controlled weakly singular Volterra delay integral equations }
		
		

     	\author[]{Jasarat J. Gasimov}
		\ead{jasarat.gasimov@emu.edu.tr}
		\author[]{Javad A. Asadzade}
		\ead{javad.asadzade@emu.edu.tr}
		\author[]{Nazim I. Mahmudov}
		\ead{nazim.mahmudov@emu.edu.tr}
		\cortext[cor1]{Corresponding author}

		\address{Department of Mathematics, Eastern Mediterranean University, Mersin 10, 99628, T.R. North Cyprus, Turkey}
		\address{	Research Center of Econophysics, Azerbaijan State University of Economics (UNEC), Istiqlaliyyat Str. 6, Baku 1001, Azerbaijan}
		

		\begin{abstract}
			
		In this article, we explore two distinct issues. Initially, we examine the utilization of the Pontriagin maximum principle in relation to fractional delay differential equations. Additionally, we discuss the optimal approach for solving the control problem for equation (1.1) and its associated payoff function (1.2). Following that, we investigate the application of the Pontryagin Maximum principle in the context of Volterra delay integral equations (1.3). We strength the results of our study by providing illustrative examples at the end of the article.
			\noindent
		\end{abstract}

	\end{frontmatter}

	\section{INTRODUCTION}
	
	\label{Sec:intro}
The Pontryagin Maximum Principle is a powerful tool in optimal control theory that is used to derive necessary conditions for finding optimal control trajectories. This principle was first introduced by Lev S. Pontryagin in 1956, and it is widely used in various mathematical and applied fields.

In recent years, the Pontryagin Maximum Principle has been extended to fractional delay differential equations (FDDEs) and delayed Volterra integral equations (DVIEs). These equations are widely used in various areas of science and engineering, including control systems, ecology, finance, biology, physics, and chemistry.

The Caputo FDDE is a type of fractional differential equation that contains a time delay term. It describes the dynamics of systems that exhibit memory effects, and it has been used to model various phenomena such as heat conduction, viscoelasticity, and fractional order control systems. The delayed Volterra integral equation, on the other hand, is a type of integral equation that contains a time delay term. It describes the behavior of systems that exhibit history-dependent dynamics, and it has been used to model various phenomena such as population dynamics, chemical reaction kinetics, and electrical networks.

The Pontryagin Maximum Principle for Caputo FDDEs and DVIEs is a set of necessary conditions that must be satisfied by optimal control trajectories for these equations. It provides a method to solve optimization problems involving these equations, such as finding the control inputs that minimize the energy consumption or maximize the system performance.

The necessary conditions derived from the Pontryagin Maximum Principle for Caputo FDDEs and DVIEs involve the optimal control trajectory, the corresponding adjoint function, and other variables that depend on the specific problem being considered. These conditions are often expressed in terms of differential or integral equations, and they can be used to derive optimal control strategies for a wide range of systems.

The article deals with two issues: Firstly, let us consider the following Pontriagin Maximum Principle for delayed differential equation. 
\begin{equation}
	\begin{cases}
		{^{C}}D^{\alpha}_{t}y(t)= f(t,y(t),y(t-h),u(t)), \quad a.e.\quad  t\in [0,T],\\
		y(t)=0,\quad -h\leq t\leq 0,\quad h>0.
	\end{cases}
\end{equation}
\begin{equation}
	J(u(\cdot))=\int_{0}^{T} g(t,y(t),y(t-h),u(t))dt.
\end{equation}
This article discusses the best way to solve the control problem for equation (1.1) and its corresponding payoff function (1.2).

Afterwards, it examines the application of the Pontryagin Maximum principle for Volterra delay integral equations.

\begin{equation}
	\begin{cases}
		y(t)=\eta(t)+\int_{0}^{t}\frac{f(t,s,y(s),y(s-h),u(s))}{(t-s)^{1-\alpha}}ds , \quad a.e.\quad  t\in [0,T],\\
		y(t)=0,\quad -h\leq t\leq 0,\quad h>0.
	\end{cases}
\end{equation}

In the overhead, $\eta(\cdot)$ and $f(\cdot, \cdot, \cdot, \cdot,\cdot)$ are given functions, called the free term and the generator of the state equation, respectively, $y(\cdot)$ is called the state trajectory taking values in the Euclidean space $R^{n}$, and $u(\cdot)$ is called the control taking values in some separable metric space $U$.  In order to evaluate the effectiveness of the control, we establish a payoff functional
\begin{equation}
	J(u(\cdot))=\int_{0}^{T} g(t,y(t),y(t-h),u(t))dt.
\end{equation}
with a term on the right hand expressing the running cost.

In article [\cite{1}], Lin and Yong addressed the issue of a state equation without delay, focusing on a cost functional that encompasses both running costs and pre-specified instant costs. However, in our paper, we extend this framework by considering a state equation generator that incorporates delay variables. Furthermore, our proposed cost functional not only includes running costs but also incorporates delay variables into these running costs.

Previous studies have focused on investigating optimal control problems for delay systems in classical settings (\cite{21},\cite{28},\cite{29}). Additionally, the application of the Pontryagin Maximum Principle has been explored for fractional differential equations (\cite{2},\cite{3},\cite{7},\cite{11},\cite{12},\cite{15}, \cite{16},\cite{25},\cite{26}), stochastic cases (\cite{22}-\cite{24},\cite{27}), and Volterra-type integral equations (\cite{1},\cite{4}-\cite{6},\cite{8}-\cite{10},\cite{13},\cite{14},\cite{18},\cite{20})
	\section{Preliminaries}
	In the following passage, we will share initial discoveries that will be helpful for future investigations. Firstly, we will focus on a specific period of time labeled as $T>0$. Then, we will establish certain spaces:
	\begin{align*}
		L^{p}(0,T;R^{n})=\bigg\lbrace \phi :[0,T]\to R^{n} \quad | \quad \phi(\cdot) \quad is \quad measurable, \\
		\Vert \phi(\cdot)\Vert_{p}\equiv \bigg(\int_{0}^{T} \vert \phi(t)\vert^{p}dt\bigg)^{1/p}<\infty\bigg\rbrace, 1\leq p<\infty,
	\end{align*}
	\begin{align*}
		L^{\infty}(0,T;R^{n})=\bigg\lbrace \phi :[0,T]\to R^{n} \quad | \quad \phi(\cdot) \quad is \quad measurable, \quad \Vert \phi(\cdot)\Vert_{\infty}\equiv \esssup_{t\in[0,T]}\vert \phi(t)\vert<\infty\bigg\rbrace.
	\end{align*}
	Also, we define 
	\begin{align*}
		L^{p+}(0,T;R^{n})=\bigcup_{r>p} L^{r}(0,T;R^{n}), \quad  1\leq p<\infty,\\
		L^{p-}(0,T;R^{n})=\bigcap_{r<p} L^{r}(0,T;R^{n}), \quad  1<p<\infty.
	\end{align*}
	In the subsequent analysis, we utilize the notation $\Delta=\lbrace (t,s)\in [0,T]^{2} | 0\leq s<t\leq T\rbrace$. It is important to note that the "diagonal line" represented by $\lbrace (t,t)| t\in [0,T]\rbrace$ does not belong to $\Delta$. Consequently, if we consider a continuous mapping $\phi :\Delta\to R^{n}$ where $(t,s)\mapsto\phi(t,s)$, the function $\phi(\cdot,\cdot)$ may become unbounded as the difference $\vert t-s\vert \to 0$.
	
	Throughout this paper, we adopt the notation $t_{1} \vee t_{2}=\max \lbrace t_{1},t_{2}\rbrace$ and $t_{1} \wedge t_{2}=\min \lbrace t_{1},t_{2}\rbrace$, for any $t_{1},t_{2}\in R$. Notably, $t^{+}=t\vee 0$. The characteristic  function of any set $E$ is denoted by $1_{E(\cdot)}$.
	
	We call a continuous  and strictly increasing function $\omega(\cdot): R_{+}\to R_{+}$ a modulus of continuity if $\omega(0)=0$.
	
\begin{definition}(\cite{19})
	The fractional integral of order $\alpha>0$ for a function $\phi :[0,T]\to R$ is defined by
	\begin{align*}
		(I^{\alpha}_{0+}\phi)(t)=\frac{1}{\Gamma(\alpha)}\int_{0}^{t}(t-\tau)^{\alpha-1}\phi(\tau)d\tau,\quad t>0.
	\end{align*}
where $\Gamma: (0,\infty)\to R$ is the well-known Euler gamma function defined as
	\begin{align*}
	\Gamma(\alpha)=\int_{0}^{\infty}\tau^{\alpha-1}e^{-\tau}d\tau, \quad \alpha>0.
\end{align*}
\end{definition}
\begin{definition}(\cite{19})
	The Riemann-Liouville fractional derivative of the order $0<\alpha\leq1$ for a function $\phi(\cdot)\in L^{\infty}[0,T]$ is defined by
	\begin{align*}
		({^{RL}} D^{\alpha}_{0+}\phi)(t)=\frac{1}{\Gamma(1-\alpha)}\frac{d}{dt}\int_{0}^{t}(t-\tau)^{-\alpha}\phi(\tau)d\tau,\quad t>0.
	\end{align*}
\end{definition}

\begin{definition}(\cite{19})
	The Caputo fractional derivative of the order $0<\alpha\leq1$ for a function $\phi(\cdot)\in L^{\infty}[0,T]$ is defined by
	\begin{align*}
			({^{C}} D^{\alpha}_{0+}\phi)(t)=\frac{1}{\Gamma(1-\alpha)}\int_{0}^{t}(t-\tau)^{-\alpha}\phi^{\prime}(\tau)d\tau,\quad t>0.
	\end{align*}
\end{definition}
The relationship Caputo and Riemann-Liouville fractional differentiation operators for the function $\phi(\cdot)\in L^{\infty}[0,T]$ is as follows:
\begin{align*}
	{^{C}} D^{\alpha}_{0+}\phi(t)={^{RL}} D^{\alpha}_{0+}(\phi(t)-\phi(0)),\quad \alpha \in (0,1].
\end{align*}
	First, we will denote the following necessary theorem, which we use it for proof of the lemma.
	\begin{lemma}(\cite{30})
		Let $\alpha\in(0,1)$ and $q>\frac{1}{\alpha}$. Let $L(\cdot), a(\cdot),y(\cdot)$ be nonnegative functions with $L(\cdot)\in L^{q}(0,T)$ and $a(\cdot),y(\cdot)\in L^{\frac{q}{q-1}}(0,T)$. Assume
		\begin{equation}
			y(t)\leq a(t)+\int_{0}^{t}\frac{L(s)y(s)}{(t-s)^{1-\alpha}}ds+\int_{0}^{t}\frac{L(s)y(s-h)}{(t-s)^{1-\alpha}}ds, \quad a.e. \quad t\in[0,T].
		\end{equation}
		Then there is a constant $K>0$ such that
		\begin{equation}
			y(t)\leq K A(t)\exp\bigg[\int_{0}^{t}L(s)ds+\int_{0}^{t}\gamma(s)ds\bigg], \quad 0\leq t<T,
		\end{equation}
	where 
	\begin{align*}
		A(t)=&a(t)+K_{0}\int_{0}^{t}\frac{L(s)a(s)}{(t-s)^{1-\alpha}}ds+K_{0}\int_{0}^{t}\frac{L(s)a(s-h)}{(t-s)^{1-\alpha}}ds+ K_{1}\int_{0}^{t-h}\frac{L(s)a(s)}{(t-h-s)^{1-\alpha}}ds\\
		+& K_{1}\int_{0}^{t-h}\frac{L(s)a(s-h)}{(t-h-s)^{1-\alpha}}ds
		+ K_{2}\int_{0}^{t-2h}\frac{L(s)a(s)}{(t-2h-s)^{1-\alpha}}ds+ K_{2}\int_{0}^{t-2h}\frac{L(s)a(s-h)}{(t-2h-s)^{1-\alpha}}ds\\
		+&\cdots+K_{n-1}\int_{0}^{t-(n-1)h}\frac{L(s)a(s)}{(t-(n-1)h-s)^{1-\alpha}}ds+ K_{n-1}\int_{0}^{t-(n-1)h}\frac{L(s)a(s-h)}{(t-(n-1)h-s)^{1-\alpha}}ds.\\
	\end{align*}
	\end{lemma}
Let  for any $\delta>0$, 
\begin{align*}
	\mathscr{E}_{\delta}=\lbrace E\subseteq [0,T] | \vert E\vert =\delta T\rbrace,
\end{align*}
where $\vert E\vert$ stands for the Lebesgue measure of E.
\begin{lemma}(\cite{1})
	Let $\varphi: \Delta \to R^{n}$ be measurable such that
	\begin{align}
\begin{cases}
	\vert \varphi(0,s)\vert\leq \bar{\varphi}(s),\quad 0<s\leq T,\\
	\vert \varphi(t,s)-\varphi(t^{\prime},s)\vert\leq\omega(\vert t-t^{\prime}\vert) \bar{\varphi}(s), \quad (t,s), (t^{\prime},s)\in \Delta.
\end{cases}
\end{align}
for some $\varphi(\cdot)\in L^{q}(0,T)$ with $q>\frac{1}{\alpha}, \alpha\in (0,1)$, and some modulus of continuity $\omega :R_{+}\to R_{+}$. Then 
\begin{align}
	\inf_{E\in \mathscr{E}_{\delta}} \sup_{t\in [0,T]} \bigg\vert \int_{0}^{t}\bigg(\frac{1}{\delta}1_{E}(s)-1\bigg)\frac{\varphi(t,s)}{(t-s)^{1-\alpha}}ds\bigg\vert =0.
\end{align}
\end{lemma}
\begin{lemma}(\cite{19})
	Let $\alpha\in(0,1,)\quad p\geq1, q\geq1 $ and $\frac{1}{p}+\frac{1}{q}\leq 1+\alpha.$ (If $ \frac{1}{p}+\frac{1}{q}=\alpha+1,$ then $p\not=1 $ and $q\not=1).$If $f(\cdot)\in I^{\alpha}_{b-}(L^{p})$ and $g(\cdot)\in I^{\alpha}_{a+}(L^{q}),$ then 
	\begin{align*}
		\int_{a}^{b}f(t)(D^{\alpha}_{a+}g)(t)dt=\int_{a}^{b}g(t)(D^{\alpha}_{b-}f)(t)dt.
	\end{align*}
\end{lemma}
	\section{Main results}
Suppose that $U$, a separable metric space with the metric $\rho$. $U$ can be either a nonempty bounded or unbounded set in $R^{m}$ with the metric induced by the usual Euclidean norm. We can view $U$ as a measurable space by considering the Borel $\sigma$-field. Let $u_{0}$ be a fixed element in $U$. For any $p \geq 1$, we define
	\begin{align*}
		U^{p}[0,T]=\lbrace u:[0,T]\to U | u(\cdot) \quad is\quad  measurable\quad  \rho(u(\cdot),u_{0})\in L^{p}(0,T;R)\rbrace.
	\end{align*}
	\textbf{\subsection{\textbf{Pontryagin's maximum principle for the fractional differential equation}}}
	Let us consider the following problem 
	\begin{equation}
		\begin{cases}
			{^{C}}D^{\alpha}_{t}y(t)= f(t,y(t),y(t-h),u(t)), \quad a.e.\quad  t\in [0,T],\\
			y(t)=0,\quad -h\leq t\leq 0,\quad h>0.
		\end{cases}
	\end{equation}
	\begin{equation}
		J(u(\cdot))=\int_{0}^{T} g(t,y(t),y(t-h),u(t))dt.
	\end{equation}
In this paragraph, we negotiate the optimal control problem for (3.1) with payoff functional (3.2) Let us acquaint the following hypothesises. The assumed conditions are more than sufficient.  Without loss generality, we state to utilize these conditions. 

(H1): Let $f: [0,T]\times R^{n}\times R^{n}\times U\rightarrow R^{n}$  be  a transformation with $t\mapsto f(t,y,y_{h},u)$ being measurable, $(y,y_{h})\mapsto f(t,y,y_{h},u)$ being continuously differentiable,  $(y,y_{h},u)\mapsto f(t,y,y_{h},u)$ , $(y,y_{h},u)\mapsto f_{y}(t,y,y_{h},u)$ and  $(y,y_{h},u)\mapsto f_{y_{h}}(t,y,y_{h},u)$being continuous. There exist non-negative functions. $L_{0}(\cdot), L(\cdot)$ with 
\begin{align}
&L_{0}(\cdot)\in L^{\frac{1}{\alpha}{+}}(0,T),\quad L(\cdot)\in L^{\frac{p}{p\alpha-1}{+}}(0,T)
\end{align}
for some $p>\frac{1}{\alpha}$ and $\alpha \in (0,1), u_{0}\in U$.
\begin{align}
	\vert f(t,0,0,u_{0})\vert \leq L_{0}(t), \quad t\in [0,T],
\end{align}
\begin{align}
	\vert f(t,y,y_{h},u)-f(t,y^{\prime},y^{\prime}_{h},u^{\prime})\vert
	 \leq L(t)[\vert y-y^{\prime}\vert+\vert y_{h}-y^{\prime}_{h}\vert+\rho(u,u^{\prime})],\quad t\in [0,T], \quad y,y^{\prime}, y_{h},y^{\prime}_{h}\in R^{n}, \quad u,u^{\prime}\in U.
\end{align}
We point out (3.4)-(3.5) declare
\begin{align}
	\vert  f(t,y,y_{h},u)\vert \leq L_{0}(t)+L(t)[\vert y\vert +\vert y_{h}\vert +\rho(u,u_{0})],\quad (t,y,y_{h},u)\in [0,T]\times R^{n}\times R^{n}
	\times U.
\end{align}
\begin{align}
	\vert f(t,y,y_{h},u)-f(t^{\prime},y,y_{h},u)\vert\leq K\omega (\vert t-t^{\prime}\vert)(1+\vert y\vert +\vert y_{h}\vert), \quad t,t^{\prime}\in[0,T],\quad y,y_{h}\in R^{n}, \quad u\in U,
\end{align}
for some modulus of continuity $\omega(\cdot)$.
Moreover, it is evident that $L$ is included in a smaller space, compared to the space to which $L_{0}$ belongs.

(H2): Let $g: [0,T]\times R^{n}\times R^{n}\times U\rightarrow R$ be a transformation with $t\mapsto g(t,y,y_{h},u)$ being measurable, $(y,y_{h})\mapsto g(t,y,y_{h},u)$ being continuously differentiable, and $(y,y_{h},u)\mapsto (g(t,y,y_{h},u),g_{y}(t,y,y_{h},u),g_{y_{h}}(t,y,y_{h},u))$ being continuous. There is a constant $L>0$ such that 
\begin{align*}
	\vert g(t,0,0,u)\vert +	\vert g_{y}(t,0,0,u)\vert+	\vert g_{y_{h}}(t,0,0,u)\vert\leq L, \quad (t,y,y_{h},u)\in [0,T]\times R^{n}\times R^{n}\times U.
\end{align*}
Obviously, the payoff functional (3.2) is well-defined under (H1) and (H2). Thus, we are able to consider the following optimal control problem (OCP).

\textbf{Problem (P)} Find a $u^{*}(\cdot)\in U^{p}[0,T]$ such that 
\begin{align}
	J(u^{*}(\cdot))=\inf_{u(\cdot)\in U^{p}[0,T]}J(u(\cdot)).
\end{align}
Arbitrary $u^{*}(\cdot)$ satisfying (3.8) is called an optimal control of Problem (P), the appropriate state $y^{*}(\cdot)$ is called an optimal state, and $(y^{*},u^{*})$ is called optimal pair. 

Now, Pontryagin's maximum principle can be given for the Problem (P).
\begin{theorem} Let (H1) and (H2) satisfy. Assume $(y^{*}(\cdot),u^{*}(\cdot))$  is an optimal pair of Problem(P). Then there is a solution $\varphi\in L^{\frac{p}{p-1}}(0,T;R^{n})$ of the following adjoint equation
\begin{align}
	\begin{cases}
		{^{C}}D^{\alpha}_{T}\psi(t)=-g_{y}(t,y^{*}(t),y^{*}(t-h),u^{*}(t))-\chi_{(0,T-h)}(t) g_{y_{h}}(t+h,y^{*}(t+h),y^{*}(t),u^{*}(t+h))\\
		+f_{y}(t,y^{*}(t),y^{*}(t-h),u^{*}(t))\psi(t)+\chi_{(0,T-h)}(t)f_{y_{h}}(t+h,y^{*}(t+h),y^{*}(t),u^{*}(t+h))\psi(t+h),\\
		\psi(t)=0,\quad T\leq t\leq T+h,\quad h>0,
	\end{cases}
\end{align}
such that the following maximum condition satisfies:
\begin{align}
	&\psi^{\top}(t)f(t,y^{*}(t),y^{*}(t-h),u^{*}(t))-g(t,y^{*}(t),y^{*}(t-h),u^{*}(t))\nonumber\\
	=&\max_{u\in U}\big[\psi^{\top}(t)f(t,y^{*}(t),y^{*}(t-h),u(t))-g(t,y^{*}(t),y^{*}(t-h),u(t))\big], \quad \forall u\in U, \quad a.e. \quad t\in [0,T].
\end{align}
\end{theorem}
\begin{proof} Proof of theorem consists of two steps.
	
	\textbf{step1.} \textbf{\textit{A variational inequality.}} Let $(y^{*}(\cdot),u^{*}(\cdot))$  be an optimal pair of Problem(P). Fix an arbitrary $u(\cdot)\in U^{p}[0,T]$. Insert
	\begin{align}
		\begin{cases}
			\widehat{f}(s)=f(s,y^{*}(s),y^{*}(s-h),u(s))-f(s,y^{*}(s),y^{*}(s-h),u^{*}(s))\\
			\widehat{g}(s)=g(s,y^{*}(s),y^{*}(s-h),u(s))-g(s,y^{*}(s),y^{*}(s-h),u^{*}(s)), \quad s \in [0,T].
		\end{cases}
	\end{align}
Taking
\begin{align*}
	\varphi(s)=\widehat{f}(s), \quad h(s)=\widehat{g}(s), \quad \quad s \in [0,T].
\end{align*}
and by using lemma 2.2 , accordingly, for arbitrary $\delta >0$, there exists an $E_{\delta}\in \mathcal{E_{\delta}}$ such that
	\begin{align}
	\begin{cases}
		\vert \int_{0}^{t}\bigg(\frac{1}{\delta} 1_{E_{\delta}}(s)-1\bigg)\frac{\widehat{f}(s)}{(t-s)^{1-\alpha}}ds\vert \leq \delta, \quad t\in [0,T]\\
		\vert \int_{0}^{T}\bigg(\frac{1}{\delta} 1_{E_{\delta}}(s)-1\bigg)\widehat{g}(s)ds\vert \leq \delta.
	\end{cases}
\end{align}
Signify
\begin{align}
	u^{\delta}(t)=
	\begin{cases}
		u^{*}(t),\quad t\in [0,T]\setminus E_{\delta},\\
		u(t),\quad t\in E_{\delta}.
	\end{cases}
\end{align}
Obviously, $	u^{\delta}(\cdot)\in U^{p}[0,T]$. Set $y^{\delta}(\cdot)=y(\cdot ;y_{0}, u^{\delta}(\cdot))$ be the appropriate state. Then 
	\begin{align*}
		y^{\delta}(t)-y^{*}(t)=&\frac{1}{\Gamma(\alpha)} \int_{0}^{t} \frac{f(s,y^{\delta}(s),y^{\delta}(s-h),u^{\delta}(s))-f(s,y^{*}(s),y^{*}(s-h),u^{*}(s))}{(t-s)^{1-\alpha}}ds\\
		=&\frac{1}{\Gamma(\alpha)} \int_{0}^{t} \frac{f(s,y^{\delta}(s),y^{\delta}(s-h),u^{\delta}(s))-f(s,y^{*}(s),y^{*}(s-h),u^{\delta}(s))}{(t-s)^{1-\alpha}}ds\\
		+&\frac{1}{\Gamma(\alpha)} \int_{0}^{t} \frac{f(s,y^{*}(s),y^{*}(s-h),u^{\delta}(s))-f(s,y^{*}(s),y^{*}(s-h),u^{*}(s))}{(t-s)^{1-\alpha}}ds\\
		=&\frac{1}{\Gamma(\alpha)}\int_{0}^{t} \Bigg(\frac{f^{\delta}_{y}(s)}{(t-s)^{1-\alpha}}[y^{\delta}(s)-y^{*}(s)]+\frac{f^{\delta}_{y_{h}}(s)}{(t-s)^{1-\alpha}}[y^{\delta}(s-h)-y^{*}(s-h)]\\
		+&1_{E_{\delta}}(s)\frac{\widehat{f}(s)}{(t-s)^{1-\alpha}}\Bigg)ds, \quad t\in [0,T].
	\end{align*}
where $\widehat{f}(\cdot)$ is denoted the system of (3.11), and
\begin{align*}
	f^{\delta}_{y}(s)=\int_{0}^{1} f_{y}(s,y^{*}(s)+\tau(y^{\delta}(s)-y^{*}(s),y^{*}(s-h)+\tau(y^{\delta}(s-h)-y^{*}(s-h)),u^{\delta}(s))d\tau\\
		f^{\delta}_{y_{h}}(s)=\int_{0}^{1} f_{y_{h}}(s,y^{*}(s)+\tau(y^{\delta}(s)-y^{*}(s),y^{*}(s-h)+\tau(y^{\delta}(s-h)-y^{*}(s-h)),u^{\delta}(s))d\tau
\end{align*}
By using (H1), we obtain 
\begin{align*}
	&\vert\widehat{f}(s)\vert=\vert f(s,y^{*}(s),y^{*}(s-h),u(s))-f(s,y^{*}(s),y^{*}(s-h),u^{*}(s))\vert\\
	\leq&2 L_{0}(s)+L(s)[2\vert y^{*}(s)\vert+2\vert y^{*}(s-h)\vert+\rho(u(s),u_{0})+\rho(u^{*}(s),u_{0})]\equiv\varphi(s), \quad s\in [0,T],\\
	&\vert f^{\delta}_{y}(s)\vert\leq L(s), \quad \vert	f^{\delta}_{y_{h}}(s)\vert\leq L(s), \quad s\in [0,T].
\end{align*}
Obviously, $\varphi(\cdot)\in L^{q}(0,T)$ for some $q\in (\frac{1}{\alpha},p)$. We get
\begin{align*}
\vert	y^{\delta}(t)-y^{*}(t)\vert\leq &\frac{1}{\Gamma(\alpha)}\int_{0}^{t} \frac{L(s)}{(t-s)^{1-\alpha}}\vert y^{\delta}(s)-y^{*}(s)\vert ds+\frac{1}{\Gamma(\alpha)}\int_{0}^{t}\frac{L(s)}{(t-s)^{1-\alpha}}\vert y^{\delta}(s-h)-y^{*}(s-h)\vert ds\\
+&\frac{1}{\Gamma(\alpha)}\int_{0}^{t}1_{E_{\delta}}(s)\frac{\varphi(s)}{(t-s)^{1-\alpha}}ds, \quad t\in [0,T].
\end{align*}
By applying the Gronwall's inequality (Lemma 2.1), selecting $q^{\prime}\in (\frac{1}{\alpha},q)$, and using Hölder inequality,\\ we have
\begin{align*}
\vert	y^{\delta}(t)-y^{*}(t)\vert\leq &\int_{0}^{t}1_{E_{\delta}}(s)\frac{\varphi(s)}{(t-s)^{1-\alpha}}ds+K\int_{0}^{t} \frac{L(s)}{(t-s)^{1-\alpha}}\int_{0}^{s}1_{E_{\delta}}(\tau)\frac{\varphi(\tau)}{(s-\tau)^{1-\alpha}}d\tau ds\\
+&K\int_{0}^{t} \frac{L(s)}{(t-s)^{1-\alpha}}\int_{0}^{s-h}1_{E_{\delta}}(\tau)\frac{\varphi(\tau)}{(s-h-\tau)^{1-\alpha}}d\tau ds\\
+&K\int_{0}^{t-h} \frac{L(s)}{(t-h-s)^{1-\alpha}}\int_{0}^{s}1_{E_{\delta}}(\tau)\frac{\varphi(\tau)}{(s-\tau)^{1-\alpha}}d\tau ds\\
+&K\int_{0}^{t-h} \frac{L(s)}{(t-h-s)^{1-\alpha}}\int_{0}^{s-h}1_{E_{\delta}}(\tau)\frac{\varphi(\tau)}{(s-h-\tau)^{1-\alpha}}d\tau ds\\
+&\cdots+K\int_{0}^{t-(n-1)h} \frac{L(s)}{(t-(n-1)h-s)^{1-\alpha}}\int_{0}^{s}1_{E_{\delta}}(\tau)\frac{\varphi(\tau)}{(s-\tau)^{1-\alpha}}d\tau ds\\
+&K\int_{0}^{t-(n-1)h} \frac{L(s)}{(t-(n-1)h-s)^{1-\alpha}}\int_{0}^{s-h}1_{E_{\delta}}(\tau)\frac{\varphi(\tau)}{(s-h-\tau)^{1-\alpha}}d\tau ds\\
\leq&\bigg(\int_{0}^{t}1_{E_{\delta}}(s)^{q^{\prime}} \varphi(s)^{q^{\prime}}ds\bigg)^{\frac{1}{q^{\prime}}}+K\int_{0}^{t}\frac{L(s)}{(t-s)^{1-\alpha}}\bigg(\int_{0}^{s}1_{E_{\delta}}(\tau)^{q^{\prime}} \varphi(\tau)^{q^{\prime}}d\tau\bigg)^{\frac{1}{q^{\prime}}}ds\\
+&K\int_{0}^{t}\frac{L(s)}{(t-s)^{1-\alpha}}\bigg(\int_{0}^{s-h}1_{E_{\delta}}(\tau)^{q^{\prime}} \varphi(\tau)^{q^{\prime}}d\tau\bigg)^{\frac{1}{q^{\prime}}}ds\\
+&K\int_{0}^{t-h}\frac{L(s)}{(t-h-s)^{1-\alpha}}\bigg(\int_{0}^{s}1_{E_{\delta}}(\tau)^{q^{\prime}} \varphi(\tau)^{q^{\prime}}d\tau\bigg)^{\frac{1}{q^{\prime}}}ds\\
+&K\int_{0}^{t-h}\frac{L(s)}{(t-h-s)^{1-\alpha}}\bigg(\int_{0}^{s-h}1_{E_{\delta}}(\tau)^{q^{\prime}} \varphi(\tau)^{q^{\prime}}d\tau\bigg)^{\frac{1}{q^{\prime}}}ds\\
+&\cdots+K\int_{0}^{t-(n-1)h}\frac{L(s)}{(t-(n-1)h-s)^{1-\alpha}}\bigg(\int_{0}^{s}1_{E_{\delta}}(\tau)^{q^{\prime}} \varphi(\tau)^{q^{\prime}}d\tau\bigg)^{\frac{1}{q^{\prime}}}ds\\
+&K\int_{0}^{t-(n-1)h}\frac{L(s)}{(t-(n-1)h-s)^{1-\alpha}}\bigg(\int_{0}^{s-h}1_{E_{\delta}}(\tau)^{q^{\prime}} \varphi(\tau)^{q^{\prime}}d\tau\bigg)^{\frac{1}{q^{\prime}}}ds\\
\leq& K^{\prime}\vert E_{\delta}\vert^{\frac{q-q^{\prime}}{q^{\prime}q}}\leq K^{\prime}\delta^{\frac{q-q^{\prime}}{q^{\prime}q}}\to 0
\end{align*}
as $\delta \to 0$, uniformly in $t\in [0,T]$. Denote
\begin{align*}
	Y^{\delta}(t)=\frac{y^{\delta}(t)-y^{*}(t)}{\delta}, \quad t\in [0,T].
\end{align*}
It follows $Y^{\delta}(\cdot)$ be a solution of the following :
\begin{align*}
		Y^{\delta}(t)=&\frac{1}{\Gamma(\alpha)}\int_{0}^{t}\bigg[ \frac{f^{\delta}_{y}(s)}{(t-s)^{1-\alpha}}Y^{\delta}(s)+\frac{f^{\delta}_{y_{h}}(s)}{(t-s)^{1-\alpha}}Y^{\delta}(s-h)+\frac{1_{E_{\delta}}(s)}{\delta}\frac{\widehat{f}(s)}{(t-s)^{1-\alpha}}\bigg]ds, \quad t\in [0,T].
\end{align*}
Let be a solution of the following variational equation
\begin{align}
	Y(t)=&\frac{1}{\Gamma(\alpha)}\int_{0}^{t}\bigg[ \frac{f_{y}(s,y^{*}(s),y^{*}(s-h),u^{*}(s))}{(t-s)^{1-\alpha}}Y(s)+\frac{f_{y_{h}}(s,y^{*}(s),y^{*}(s-h),u^{*}(s))}{(t-s)^{1-\alpha}}Y(s-h)+\frac{\widehat{f}(s)}{(t-s)^{1-\alpha}}\bigg]ds
\end{align}
We get
\begin{align*}
	\vert 	Y(t)\vert \leq&\frac{1}{\Gamma(\alpha)}\int_{0}^{t}\frac{\varphi(s)}{(t-s)^{1-\alpha}}ds+\frac{1}{\Gamma(\alpha)}\int_{0}^{t}\frac{L(s)}{(t-s)^{1-\alpha}}\vert Y(s)\vert ds+\frac{1}{\Gamma(\alpha)}\int_{0}^{t}\frac{L(s)}{(t-s)^{1-\alpha}}\vert Y(s-h)\vert ds, \quad t\in [0,T].
\end{align*}
By applying the Gronwall's inequality (Lemma 2.1), mention $q>\frac{1}{\alpha}$ and $L(\cdot)\in L^{\frac{1}{\alpha}+}(0,T)$,
\begin{align*}
		\vert 	Y(t)\vert \leq&\frac{1}{\Gamma(\alpha)}\int_{0}^{t}\frac{\varphi(s)}{(t-s)^{1-\alpha}}ds+\frac{K}{(\Gamma(\alpha))^{2}}\int_{0}^{t}\frac{L(s)}{(t-s)^{1-\alpha}}\int_{0}^{s}\frac{\varphi(\tau)}{(s-\tau)^{1-\alpha}}d\tau ds\\
		+&\frac{K}{(\Gamma(\alpha))^{2}}\int_{0}^{t}\frac{L(s)}{(t-s)^{1-\alpha}}\int_{0}^{s-h}\frac{\varphi(\tau)}{(s-h-\tau)^{1-\alpha}}d\tau ds\\
		+&\frac{K}{(\Gamma(\alpha))^{2}}\int_{0}^{t-h}\frac{L(s)}{(t-h-s)^{1-\alpha}}\int_{0}^{s}\frac{\varphi(\tau)}{(s-\tau)^{1-\alpha}}d\tau ds\\
		+&\frac{K}{(\Gamma(\alpha))^{2}}\int_{0}^{t-h}\frac{L(s)}{(t-h-s)^{1-\alpha}}\int_{0}^{s-h}\frac{\varphi(\tau)}{(s-h-\tau)^{1-\alpha}}d\tau ds\\
		+&\cdots+\frac{K}{(\Gamma(\alpha))^{2}}\int_{0}^{t-(n-1)h}\frac{L(s)}{(t-(n-1)h-s)^{1-\alpha}}\int_{0}^{s}\frac{\varphi(\tau)}{(s-\tau)^{1-\alpha}}d\tau ds\\
		+&\frac{K}{(\Gamma(\alpha))^{2}}\int_{0}^{t-(n-1)h}\frac{L(s)}{(t-(n-1)h-s)^{1-\alpha}}\int_{0}^{s-(n-1)h}\frac{\varphi(\tau)}{(s-h-\tau)^{1-\alpha}}d\tau ds\\
		\leq& K^{\prime}\Vert \varphi(\cdot)\Vert_{q}, \quad t\in[0,T]. 
\end{align*}
Therefore, 
\begin{align*}
	Y^{\delta}(t)-Y(t)=&\frac{1}{\Gamma(\alpha)}\int_{0}^{t}\bigg[ \frac{f^{\delta}_{y}(s)}{(t-s)^{1-\alpha}}Y^{\delta}(s)-\frac{f_{y}(s,y^{*}(s),y^{*}(s-h),u^{*}(s))}{(t-s)^{1-\alpha}}Y(s)\bigg]ds\\
	+&\frac{1}{\Gamma(\alpha)}\int_{0}^{t}\bigg[\frac{f^{\delta}_{y_{h}}(s)}{(t-s)^{1-\alpha}}Y^{\delta}(s-h)-\frac{f_{y_{h}}(s,y^{*}(s),y^{*}(s-h),u^{*}(s))}{(t-s)^{1-\alpha}}Y(s-h)\bigg]ds\\
	+&\frac{1}{\Gamma(\alpha)}\int_{0}^{t}\bigg(\frac{1_{E_{\delta}}(s)}{\delta}-1\bigg)\frac{\widehat{f}(s)}{(t-s)^{1-\alpha}}ds=\frac{1}{\Gamma(\alpha)}\int_{0}^{t}\frac{f^{\delta}_{y}(s)}{(t-s)^{1-\alpha}}[Y^{\delta}(s)-Y(s)]ds\\
	+&\frac{1}{\Gamma(\alpha)}\int_{0}^{t}\frac{f^{\delta}_{y}(s)-f_{y}(s,y^{*}(s),y^{*}(s-h),u^{*}(s))}{(t-s)^{1-\alpha}}Y(s)ds+\int_{0}^{t}\frac{f^{\delta}_{y_{h}}(s)}{(t-s)^{1-\alpha}}[Y^{\delta}(s-h)-Y(s-h)]ds\\
	+&\frac{1}{\Gamma(\alpha)}\int_{0}^{t}\frac{f^{\delta}_{y_{h}}(s)-f_{y_{h}}(s,y^{*}(s),y^{*}(s-h),u^{*}(s))}{(t-s)^{1-\alpha}}Y(s-h)ds+\int_{0}^{t}\bigg(\frac{1_{E_{\delta}}(s)}{\delta}-1\bigg)\frac{\widehat{f}(s)}{(t-s)^{1-\alpha}}ds\\
	\equiv&\frac{1}{\Gamma(\alpha)}\int_{0}^{t}\frac{f^{\delta}_{y}(s)}{(t-s)^{1-\alpha}}[Y^{\delta}(s)-Y(s)]ds+\frac{1}{\Gamma(\alpha)}\int_{0}^{t}\frac{f^{\delta}_{y_{h}}(s)}{(t-s)^{1-\alpha}}[Y^{\delta}(s-h)-Y(s-h)]ds\\
	+&a^{\delta}_{1}(t)+a^{\delta}_{2}(t)+a^{\delta}_{3}(t), \quad t\in [0,T].
\end{align*}
Since
\begin{align*}
&\Big\vert	\frac{f^{\delta}_{y}(s)-f_{y}(s,y^{*}(s),y^{*}(s-h),u^{*}(s))}{(t-s)^{1-\alpha}}Y(s)\Big\vert\leq\frac{2L(s)}{(t-s)^{1-\alpha}}\vert Y(s)\vert, \quad a.e.\quad s\in[0,t),\\
&\Big\vert\frac{f^{\delta}_{y_{h}}(s)-f_{y_{h}}(s,y^{*}(s),y^{*}(s-h),u^{*}(s))}{(t-s)^{1-\alpha}}Y(s-h)\Big\vert\leq\frac{2L(s)}{(t-s)^{1-\alpha}}\vert Y(s-h)\vert, \quad a.e.\quad s\in[0,t).
\end{align*}
with
\begin{align*}
	\int_{0}^{t}\frac{2L(s)}{(t-s)^{1-\alpha}}\vert Y(s)\vert ds<\infty,\\
	\int_{0}^{t}\frac{2L(s)}{(t-s)^{1-\alpha}}\vert Y(s-h)\vert ds<\infty.
\end{align*}
By using the dominated convergence theorem, we get
\begin{align*}
	\lim_{\delta\to 0}a^{\delta}_{1}(t)=0, \quad \lim_{\delta\to 0}a^{\delta}_{2}(t)=0, \quad t\in [0,T].
\end{align*}
In addition, by (3.12),
\begin{align*}
	a^{\delta}_{3}(t)\equiv \Big\vert \frac{1}{\Gamma(\alpha)}\int_{0}^{t}\bigg(\frac{1_{E_{\delta}}(s)}{\delta}-1\bigg)\frac{\widehat{f}(s)}{(t-s)^{1-\alpha}}ds  \Big\vert\leq \delta, \quad t\in [0,T].
\end{align*}
Therefore,
\begin{align*}
	\vert Y^{\delta}(t)-Y(t) \vert\leq&\frac{1}{\Gamma(\alpha)}\int_{0}^{t}\frac{L(s)}{(t-s)^{1-\alpha}}\vert Y^{\delta}(s)-Y(s)\vert ds+ \frac{1}{\Gamma(\alpha)}\int_{0}^{t}\frac{L(s)}{(t-s)^{1-\alpha}}\vert Y^{\delta}(s-h)-Y(s-h)\vert ds\\
	+&\vert a^{\delta}_{1}(t)\vert+\vert a^{\delta}_{2}(t)\vert+\vert a^{\delta}_{3}(t)\vert, \quad t\in [0,T].
\end{align*}
Applying the Gronwall's inequality,
\begin{align*}
	\vert Y^{\delta}(t)-Y(t) \vert\leq&\vert a^{\delta}_{1}(t)\vert+\vert a^{\delta}_{2}(t)\vert+\vert a^{\delta}_{3}(t)\vert+\frac{K}{\Gamma(\alpha)}\int_{0}^{t}\frac{L(s)}{(t-s)^{1-\alpha}}(\vert a^{\delta}_{1}(s)\vert+\vert a^{\delta}_{2}(s)\vert+\vert a^{\delta}_{3}(s)\vert) ds\\
	+&\frac{K}{\Gamma(\alpha)} \int_{0}^{t}\frac{L(s)}{(t-s)^{1-\alpha}}(\vert a^{\delta}_{1}(s-h)\vert+\vert a^{\delta}_{2}(s-h)\vert+\vert a^{\delta}_{3}(s-h)\vert) ds\\
	+&\frac{K}{\Gamma(\alpha)}\int_{0}^{t-h}\frac{L(s)}{(t-h-s)^{1-\alpha}}(\vert a^{\delta}_{1}(s)\vert+\vert a^{\delta}_{2}(s)\vert+\vert a^{\delta}_{3}(s)\vert) ds\\
	+&\frac{K}{\Gamma(\alpha)} \int_{0}^{t-h}\frac{L(s)}{(t-h-s)^{1-\alpha}}(\vert a^{\delta}_{1}(s-h)\vert+\vert a^{\delta}_{2}(s-h)\vert+\vert a^{\delta}_{3}(s-h)\vert) ds\\
	+&\cdots +\frac{K}{\Gamma(\alpha)}\int_{0}^{t-(n-1)h}\frac{L(s)}{(t-(n-1)h-s)^{1-\alpha}}(\vert a^{\delta}_{1}(s)\vert+\vert a^{\delta}_{2}(s)\vert+\vert a^{\delta}_{3}(s)\vert) ds\\
	+&\frac{K}{\Gamma(\alpha)} \int_{0}^{t-(n-1)h}\frac{L(s)}{(t-(n-1)h-s)^{1-\alpha}}(\vert a^{\delta}_{1}(s-h)\vert+\vert a^{\delta}_{2}(s-h)\vert+\vert a^{\delta}_{3}(s-h)\vert) ds, \quad t\in [0,T].
\end{align*}
As a result, using the dominated convergence theorem, we get
\begin{align*}
	\lim_{\delta\to 0}	\vert Y^{\delta}(t)-Y(t) \vert=0, \quad t\in [0,T].
\end{align*}
With optimality of $(y^{*}(\cdot),u^{*}(\cdot))$, we have
\begin{align*}
	0\leq&\frac{J(u^{\delta}(\cdot)-J(u^{*}(\cdot)))}{\delta}=\int_{0}^{T}\big[ g^{\delta}_{y}(t)Y^{\delta}(t)+g^{\delta}_{y_{h}}(t)Y^{\delta}(t-h)+\frac{1}{\delta}1_{E_{\delta}}(t)\widehat{g}(t)\big]dt
\end{align*}
where $\widehat{g}(\cdot)$ is denoted in (3.11), and 
\begin{align*}
	g^{\delta}_{y}(t)=&\int_{0}^{1} g_{y}(t,y^{*}(t)+\tau(y^{\delta}(t)-y^{*}(t),y^{*}(t-h)+\tau(y^{\delta}(t-h)-y^{*}(t-h)),u^{\delta}(t))d\tau,\\
	g^{\delta}_{y_{h}}(t)=&\int_{0}^{1} g_{y_{h}}(t,y^{*}(t)+\tau(y^{\delta}(t)-y^{*}(t),y^{*}(t-h)+\tau(y^{\delta}(t-h)-y^{*}(t-h)),u^{\delta}(t))d\tau, \quad t\in [0,T].
\end{align*}
Therefore, by utilizing (3.12) and the convergence $y^{\delta}(t)\to y^{*}(t)$, \quad $t\in [0,T]$, it follows
\begin{align*}
	0\leq&\int_{0}^{T}\Big( g_{y}(t,y^{*}(t),y^{*}(t-h),u^{*}(t))Y(t)+g_{y_{h}}(t,y^{*}(t),y^{*}(t-h),u^{*}(t))Y(t-h)\\
	+&g(t,y^{*}(t),y^{*}(t-h),u(t))-g(t,y^{*}(t),y^{*}(t-h),u^{*}(t))\Big)dt.
\end{align*}

\textbf{step2.} \textbf{\textit{Duality.}}
Suppose $\psi(\cdot)$ be the solution of the adjoint equation (3.9). Than we get
\begin{align*}
		0\leq&\int_{0}^{T}\Big(g_{y}(s,y^{*}(s),y^{*}(s-h),u^{*}(s))Y(s)+g_{y_{h}}(s,y^{*}(s),y^{*}(s-h),u^{*}(s))Y(s-h)\\
		+&g(s,y^{*}(s),y^{*}(s-h),u(s))-g(s,y^{*}(s),y^{*}(s-h),u^{*}(s))\Big)ds\\
		=&\int_{0}^{T}\Big(g_{y}(s,y^{*}(s),y^{*}(s-h),u^{*}(s))+ \chi_{(0,T-h)}(s) g_{y_{h}}(s+h,y^{*}(s+h),y^{*}(s),u^{*}(s+h))\Big)Y(s)ds\\
		+&\int_{0}^{T}g(s,y^{*}(s),y^{*}(s-h),u(s))-g(s,y^{*}(s),y^{*}(s-h),u^{*}(s))ds\\
		=&\int_{0}^{T}\Big(-{^{C}}D^{\alpha}_{T}\psi(s)+f_{y}(s,y^{*}(s),y^{*}(s-h),u^{*}(s))\psi(s)\\
		+&\chi_{(0,T-h)}(s)f_{y_{h}}(s+h,y^{*}(s+h),y^{*}(s),u^{*}(s+h))\psi(s+h)\Big)Y(s)ds\\
		+&\int_{0}^{T}g(s,y^{*}(s),y^{*}(s-h),u(s))-g(s,y^{*}(s),y^{*}(s-h),u^{*}(s))ds\\
		=&\int_{0}^{T}\Big(-{^{C}}D^{\alpha}_{T}\psi(s)\Big)Y(s)ds+\int_{0}^{T}\Big(f_{y}(s,y^{*}(s),y^{*}(s-h),u^{*}(s))\psi(s)\\
		+&\chi_{(0,T-h)}(s)f_{y_{h}}(s+h,y^{*}(s+h),y^{*}(s),u^{*}(s+h))\psi(s+h)\Big)Y(s)ds\\
		+&\int_{0}^{T}g(s,y^{*}(s),y^{*}(s-h),u(s))-g(s,y^{*}(s),y^{*}(s-h),u^{*}(s))ds\\
		=&\int_{0}^{T}\psi(s)\Big(-{^{C}}D^{\alpha}_{0}Y(s)\Big)ds+\int_{0}^{T}f_{y}(s,y^{*}(s),y^{*}(s-h),u^{*}(s))\psi(s)Y(s)ds\\
		+&\int_{0}^{T}\chi_{(0,T-h)}(s)f_{y_{h}}(s+h,y^{*}(s+h),y^{*}(s),u^{*}(s+h))\psi(s+h)\Big)Y(s)ds\\
		+&\int_{0}^{T}g(s,y^{*}(s),y^{*}(s-h),u(s))-g(s,y^{*}(s),y^{*}(s-h),u^{*}(s))ds\\
		=&\int_{0}^{T}\psi(s)\bigg[ -{^{C}}D^{\alpha}_{0}Y(s)+f_{y}(s,y^{*}(s),y^{*}(s-h),u^{*}(s))Y(s)\bigg]ds\\
		+&\int_{h}^{T+h}\chi_{(h,T)}(s-h)f_{y_{h}}(s,y^{*}(s),y^{*}(s-h),u^{*}(s))\psi(s)Y(s-h) ds\\
		+&\int_{0}^{T}g(s,y^{*}(s),y^{*}(s-h),u(s))-g(s,y^{*}(s),y^{*}(s-h),u^{*}(s))ds\\
		=&\int_{0}^{T}\psi(s)\bigg[ -{^{C}}D^{\alpha}_{0}Y(s)+f_{y}(s,y^{*}(s),y^{*}(s-h),u^{*}(s))Y(s)\\
		+&\chi_{(h,T)}(s-h)f_{y_{h}}(s,y^{*}(s),y^{*}(s-h),u^{*}(s))Y(s-h) \bigg]ds\\
		+&\int_{0}^{T}g(s,y^{*}(s),y^{*}(s-h),u(s))-g(s,y^{*}(s),y^{*}(s-h),u^{*}(s))ds\\
		=&\int_{0}^{T}-\psi^{\top}(s)\widehat{f}(s)ds+\int_{0}^{T}\widehat{g}(s)ds=\int_{0}^{T}[\widehat{g}(s)-\psi^{\top}(s)\widehat{f}(s)]ds\equiv \int_{0}^{T}[\Psi(s,u^{*})-\Psi(s,u)]ds
\end{align*}
where
\begin{align*}
&\Psi (t,u)=\psi^{\top}(t)f(t,y^{*}(t),y^{*}(t-h),u(t))-g(t,y^{*}(t),y^{*}(t-h),u(t))
\end{align*}
and $Y(\cdot)$ is the solution of the following differential equation:
\begin{align*}
	\begin{cases}
		&{^{C}}D^{\alpha}_{0}Y(s)=f_{y}(t,y^{*}(t),y^{*}(t-h),u^{*}(t))Y(t)
		+\chi_{(h,T)}(t-h)f_{y_{h}}(t,y^{*}(t),y^{*}(t-h),u^{*}(t))Y(t-h)+\widehat{f}(t)\\
		&Y(t)=0, \quad -h\leq t\leq 0, \quad h>0.
	\end{cases}
\end{align*}
Let $u\in U$ be fixed and let $s_{0}\in(0,T)$ be a Lebesgue point of $\Psi(s,u^{*})-\Psi(s,u)$. Then for any $\delta>0$, let
\begin{align*}
	u(s)=
	\begin{cases}
		u^{*}(s),\quad s\in [0,T]\setminus (s_{0}-\delta, s_{0}+\delta),\\
	u,\quad t\in(s_{0}-\delta, s_{0}+\delta).
		\end{cases}
\end{align*}
Therefore, from the above, we get 
\begin{align*}
	0\leq \frac{1}{2\delta}=\int_{s_{0}-\delta}^{s_{0}+\delta}[\Psi(s,u^{*}(s))-\Psi(s,u)]ds\to \Psi(s_{0},u^{*}(s_{0}))-\Psi(s_{0},u),\quad \delta\to 0.
\end{align*}
Consequently, by the Lebesgue point theorem, we get the maximum condition (3.10).
\end{proof}
\bigskip
\subsection{\textbf{Maximum principle for the delayed Volterra equation with the weak singular kernel}}
\bigskip
Let us research the following issue
\begin{equation}
	\begin{cases}
	y(t)=\eta(t)+\int_{0}^{t}\frac{f(t,s,y(s),y(s-h),u(s))}{(t-s)^{1-\alpha}}ds , \quad a.e.\quad  t\in [0,T],\\
		y(t)=0,\quad -h\leq t\leq 0,\quad h>0.
	\end{cases}
\end{equation}
where $\eta(\cdot)\in L^{p}(0,T;R)$.
\begin{equation}
	J(u(\cdot))=\int_{0}^{T} g(t,y(t),y(t-h),u(t))dt.
\end{equation}
Now, we discuss the optimal control problem for (3.15) with payoff functional (3.16).  We introduce certain assumptions that are considered to be more than enough for our requirements.  Therefore, we will make use of these conditions without any loss of generality.

($H1^{\prime}$): Let $f: \Delta \times R^{n}\times R^{n}\times U\rightarrow R^{n}$  be  a transformation with $(t,s)\mapsto f(t,s,y,y_{h},u)$ being measurable, $(y,y_{h})\mapsto f(t,s,y,y_{h},u)$ being continuously differentiable,  $(y,y_{h},u)\mapsto f(t,s,y,y_{h},u)$ , $(y,y_{h},u)\mapsto f_{y}(t,s,y,y_{h},u)$ and  $(y,y_{h},u)\mapsto f_{y_{h}}(t,s,y,y_{h},u)$being continuous. Furthermore, there are non-negative functions involved in these conditions. $L_{0}(\cdot), L(\cdot)$ with 
\begin{align}
	&L_{0}(\cdot)\in L^{\frac{1}{\alpha}{+}}(0,T),\quad L(\cdot)\in L^{\frac{p}{p\alpha-1}{+}}(0,T)
\end{align}
for some $p>\frac{1}{\alpha}$ and $\alpha \in (0,1), u_{0}\in U$.
\begin{align}
	\vert f(t,s,0,0,u_{0})\vert \leq L_{0}(s), \quad (t,s)\in \Delta,
\end{align}
\begin{align}
	&\vert f(t,s,y,y_{h},u)-f(t,s,y^{\prime},y^{\prime}_{h},u^{\prime})\vert
	\leq L(s)[\vert y-y^{\prime}\vert+\vert y_{h}-y^{\prime}_{h}\vert+\rho(u,u^{\prime})],\nonumber\\
	 &(t,s)\in \Delta, \quad y,y^{\prime}, y_{h},y^{\prime}_{h}\in R^{n}, \quad u,u^{\prime}\in U.
\end{align}
We point out (3.18)-(3.19) declare
\begin{align}
	\vert  f(t,s,y,y_{h},u)\vert \leq L_{0}(s)+L(s)[\vert y\vert +\vert y_{h}\vert +\rho(u,u_{0})],\quad (t,s,y,y_{h},u)\in \Delta\times R^{n}\times R^{n}
	\times U.
\end{align}
\begin{align}
	\vert f(t,s,y,y_{h},u)-f(t^{\prime},y,y_{h},u)\vert\leq K\omega (\vert t-t^{\prime}\vert)(1+\vert y\vert +\vert y_{h}\vert), \quad (t,s),(t^{\prime},s)\in\Delta,\quad y,y_{h}\in R^{n}, \quad u\in U,
\end{align}
for some modulus of continuity $\omega(\cdot)$.
Moreover, it is evident that $L$ is included in a smaller space, compared to the space to which $L_{0}$ belongs.

($H2^{\prime}$): Let $g: [0,T]\times R^{n}\times R^{n}\times U\rightarrow R$ be a transformation with $t\mapsto g(t,y,y_{h},u)$ being measurable, $(y,y_{h})\mapsto g(t,y,y_{h},u)$ being continuously differentiable, and $(y,y_{h},u)\mapsto (g(t,y,y_{h},u),g_{y}(t,y,y_{h},u),g_{y_{h}}(t,y,y_{h},u))$ being continuous. There is a constant $L>0$ such that 
\begin{align*}
	\vert g(t,0,0,u)\vert +	\vert g_{y}(t,0,0,u)\vert+	\vert g_{y_{h}}(t,0,0,u)\vert\leq L, \quad (t,y,y_{h},u)\in [0,T]\times R^{n}\times R^{n}\times U.
\end{align*}
Given that conditions ($H1^{\prime}$) and ($H2^{\prime}$) hold, we may proceed to analyze the optimal control problem (OCP) defined by the payoff functional (3.16).

\textbf{Problem ($P^{\prime}$)} Find a $u^{*}(\cdot)\in U^{p}[0,T]$ such that 
\begin{align}
	J(u^{*}(\cdot))=\inf_{u(\cdot)\in U^{p}[0,T]}J(u(\cdot)).
\end{align}
An optimal control for Problem ($P^{\prime}$) is any arbitrary $u^{}(\cdot)$ that satisfies equation (3.22). The corresponding state $y^{}(\cdot)$ is known as an optimal state, and when combined with the optimal control $u^{}$, they form the optimal pair $(y^{},u^{*})$.

Now, Pontryagin's maximum principle can be applied to Problem ($P^{\prime}$) and can be restated as follows.
\begin{lemma}
 Let ($H1^{\prime}$) and ($H2^{\prime}$) satisfy. Then the following result holds:
	\begin{equation}
		\sup_{t\in[0,T]}\vert y^{\delta}(t)-y^{*}-\delta Y(t)\vert=o(\delta),
	\end{equation}
	where $Y$ is the solution to the first variational equation related to the optimal pair $(y^{*},u^{*}(\cdot))\in R^n\times U^p[0,T]$ given by
	\begin{align}
		Y(t)=&\int_{0}^{t}\bigg[ \frac{f_{y}(t,s,y^{*}(s),y^{*}(s-h),u^{*}(s))}{(t-s)^{1-\alpha}}Y(s)+\frac{f_{y_{h}}(t,s,y^{*}(s),y^{*}(s-h),u^{*}(s))}{(t-s)^{1-\alpha}}Y(s-h)+\frac{\widehat{f}(t,s)}{(t-s)^{1-\alpha}}\bigg]ds
	\end{align}
	with for any $u(\cdot)\in U^p[0,T],$
	\begin{align}
		\widehat{f}(t,s)=f(t,s,y^{*}(s),y^{*}(s-h),u(s))-f(t,s,y^{*}(s),y^{*}(s-h),u^{*}(s)), \quad (t,s) \in \Delta.
	\end{align}
\end{lemma}
\begin{proof}
	It is obvious from step 1 in the proof of Theorem 3.1.
\end{proof}
\begin{theorem} Let ($H1^{\prime}$) and ($H2^{\prime}$) satisfy. Assume $(y^{*}(\cdot),u^{*}(\cdot))$  is an optimal pair of Problem($P^{\prime}$). Then there is a solution $\varphi\in L^{\frac{p}{p-1}}(0,T;R^{n})$ of the following adjoint equation
	\begin{align}
			&\psi(t)=-g_{y}(t,y^{*}(t),y^{*}(t-h),u^{*}(t))^{\top}-g_{y_{h}}(t,y^{*}(t),y^{*}(t-h),u^{*}(t))^{\top}\nonumber\\
			+&\int_{s}^{T}\frac{f_{y}(t,s,y^{*}(s),y^{*}(s-h),u^{*}(s))^{\top}}{(t-s)^{1-\alpha}}\psi(t)dt+\int_{s+h}^{T}\frac{f_{y_{h}}(t,s+h,y^{*}(s+h),y^{*}(s),u^{*}(s+h))^{\top}}{(t-h-s)^{1-\alpha}}\psi(t)dt \quad a.e.\quad t\in [0,T],\\
			&\psi(t)=0, \quad T\leq t \leq T+h, \quad h>0.\nonumber
	\end{align}
	such that the following maximum condition satisfies:
	\begin{align}
		&\int_{s}^{T}\psi(t)^{\top} \frac{f(t,s,y^{*}(s),y^{*}(s-h),u^{*}(s))}{(t-s)^{1-\alpha}}dt
		-g(s,y^{*}(s),y^{*}(s-h),u^{*}(s))\nonumber\\
		=&\max_{u\in U}\big[\int_{s}^{T}\psi(t)^{\top} \frac{f(t,s,y^{*}(s),y^{*}(s-h),u(s))}{(t-s)^{1-\alpha}}dt
		-g(s,y^{*}(s),y^{*}(s-h),u(s))\big], \quad \forall u\in U, \quad a.e. \quad t\in [0,T].
	\end{align}
\end{theorem}
\begin{proof}
Based on theorem 3.1, and lemma 4.1, we have 
\begin{align*}
	0\leq&\int_{0}^{T}\Big( g_{y}(t,y^{*}(t),y^{*}(t-h),u^{*}(t))Y(t)+g_{y_{h}}(t,y^{*}(t),y^{*}(t-h),u^{*}(t))Y(t-h)\\
	+&g(t,y^{*}(t),y^{*}(t-h),u(t))-g(t,y^{*}(t),y^{*}(t-h),u^{*}(t))\Big)dt.
\end{align*}
 Similarly, assume $\psi(\cdot)$ be the solution of the adjoint equation (3.26). Than we get
\begin{align*}
	0\leq&\int_{0}^{T}\Big( g_{y}(s,y^{*}(s),y^{*}(s-h),u^{*}(s))Y(s)+g_{y_{h}}(s,y^{*}(s),y^{*}(s-h),u^{*}(s))Y(s-h)\\
	+&g(s,y^{*}(s),y^{*}(s-h),u(s))-g(s,y^{*}(s),y^{*}(s-h),u^{*}(s))\Big)ds\\
	=&\int_{0}^{T}\big( -\psi(s)+\int_{s}^{T}\frac{f_{y}(t,s,y^{*}(s),y^{*}(s-h),u^{*}(s))}{(t-s)^{1-\alpha}}\psi(t)dt\\
    +&\int_{s+h}^{T}\frac{f_{y_{h}}(t,s+h,y^{*}(s+h),y^{*}(s),u^{*}(s+h))}{(t-h-s)^{1-\alpha}}\psi(t)dt\big)^{\top}Y(s)ds\\
	+&\int_{0}^{T}g(s,y^{*}(s),y^{*}(s-h),u(s))-g(s,y^{*}(s),y^{*}(s-h),u^{*}(s))\Big)ds\\
	=&\int_{0}^{T}\psi(t)^{\top}\big(-Y(t)+\int_{0}^{t}\frac{f_{y}(t,s,y^{*}(s),y^{*}(s-h),u^{*}(s))}{(t-s)^{1-\alpha}}Y(s)ds\\
	+&\int_{0}^{t}\frac{f_{y_{h}}(t,s,y^{*}(s),y^{*}(s-h),u^{*}(s))}{(t-s)^{1-\alpha}}Y(s-h)ds\big)dt\\
	+&\int_{0}^{T}g(s,y^{*}(s),y^{*}(s-h),u(s))-g(s,y^{*}(s),y^{*}(s-h),u^{*}(s))\Big)ds\\
	=&\int_{0}^{T}\bigg[-\psi(t)^{\top} \int_{0}^{t} \frac{f(t,s,y^{*}(s),y^{*}(s-h),u(s))-f(t,s,y^{*}(s),y^{*}(s-h),u^{*}(s))}{(t-s)^{1-\alpha}}ds\bigg]dt\\
	+&\int_{0}^{T}g(s,y^{*}(s),y^{*}(s-h),u(s))-g(s,y^{*}(s),y^{*}(s-h),u^{*}(s))\Big)ds\\
	=& \int_{0}^{T}\bigg[ \int_{s}^{T}-\psi(t)^{\top}\bigg( \frac{f(t,s,y^{*}(s),y^{*}(s-h),u(s))-f(t,s,y^{*}(s),y^{*}(s-h),u^{*}(s))}{(t-s)^{1-\alpha}}\bigg)dt\\
	+&g(s,y^{*}(s),y^{*}(s-h),u(s))-g(s,y^{*}(s),y^{*}(s-h),u^{*}(s))\bigg]ds\\
	\equiv& \int_{0}^{T}[\Psi(s,u^{*})-\Psi(s,u)]ds
\end{align*}
\begin{align*}
	\Psi(s,u)=&\int_{s}^{T}\psi(t)^{\top} \frac{f(t,s,y^{*}(s),y^{*}(s-h),u(s))}{(t-s)^{1-\alpha}}dt
	-g(s,y^{*}(s),y^{*}(s-h),u(s))
\end{align*}
Let $u\in U$ be fixed and let $s_{0}\in(0,T)$ be a Lebesgue point of $\Psi(s,u^{*})-\Psi(s,u)$. Then for any $\delta>0$, let
\begin{align*}
	u(s)=
	\begin{cases}
		u^{*}(s),\quad s\in [0,T]\setminus (s_{0}-\delta, s_{0}+\delta),\\
	u,\quad t\in(s_{0}-\delta, s_{0}+\delta).
		\end{cases}
\end{align*}
Therefore, from the above, we get 
\begin{align*}
	0\leq \frac{1}{2\delta}=\int_{s_{0}-\delta}^{s_{0}+\delta}[\Psi(s,u^{*}(s))-\Psi(s,u)]ds\to \Psi(s_{0},u^{*}(s_{0}))-\Psi(s_{0},u),\quad \delta\to 0.
\end{align*}
Consequently, by the Lebesgue point theorem, we get the maximum condition (3.27).
\end{proof}
\section{Example}
Let us consider the following fractional delay optimal control problem (FDOCP).
\begin{align}
	&(D^{\alpha}_{0+}y)(t)=Ay(t-h)+Bu(t),\quad t \in [0,2] \quad a.e.\\\nonumber
	\\
	&y(t)=0,\quad -h\leq t\leq0,\quad h>0, \\\nonumber
	\\
	&u(t)\in[0,1],\quad t\in [0,2],\\\nonumber
	\\
	&J(y,u)=\int_{0}^{2}(y_{1}(t-h)-y_{2}(t-h)+u(t))dt\to \min.
\end{align}

$where  \quad y\in R^{2},\quad A=\begin{bmatrix}
	0 & 1   \\[0.3em]
	0 & 0   \\[0.3em]
\end{bmatrix}, \quad
B=\begin{bmatrix}
	-1   \\[0.3em]
	-1   \\[0.3em]
\end{bmatrix}, \quad
 \alpha=\frac{1}{2},\quad h=\frac{1}{2}.$

In this case 
\begin{align*}
	g(t,y(t),y(t-h),u(t))=Ay(t-h)+Bu(t),\\
	g_{0}(t,y(t),y(t-h),u(t))=<(1,-1),y(t-h)>+u.
\end{align*}
Clearly,
\begin{align*}
	&A^{k}=(A^{T})^{k}=0,\quad k\geq 2,\\
	&g_{y_{h}}(t,y(t),y(t-h),u)=A,\quad  (g_{0})_{y_{h}}(t,y(t),y(t-h),u)=[1,-1].
\end{align*}
Consequently, if $(y^{*},u^{*})$ is a locally optimal solution to problem (4.1)-(4.4), hence there is $\lambda$ such that
\begin{align}
	&	\big( ^{C}D^{1/2}_{0+}\lambda \big)(t)=A^{T}\lambda(t-\frac{1}{2})+\begin{bmatrix}
		-1   \\[0.3em]
		-1   \\[0.3em]
	\end{bmatrix},\\
	&\lambda(t)=0,\quad t\in\big[-\frac{1}{2}, 0\big].
\end{align}
Moreover,
\begin{align}
	u^{*}(t)-\lambda(t)Bu^{*}(t)=\min_{u\in[0,1]}\lbrace u-\lambda(t) Bu\rbrace
\end{align}
for $t\in [0,2]$ a.e.
Now, we must prove the following theorem, which will be used later on.
\begin{theorem}(\cite{31})
	Let $\alpha\in(0,1)$. Assume
	\begin{align}
		\begin{cases}
			^{C}D^{\alpha}_{0+} x(t)=Ax(t)+B x(t-h)+Cu(t),\quad t\in [0,T],\\
			x(t)=\phi(t),\quad -h < t\leq0.
		\end{cases}
	\end{align}
	where $x\in R^{n},\quad u\in R^{m}$, $A$ and $B$ are $n\times n$ matrices, $C$ is an $n\times m$ matrix with $n >m$.  The solution of (4.8) is given as
	\begin{align*}
		x(t)=X_{\alpha}(t)\phi(0)+B\int_{-h}^{0}(t-s-h)^{\alpha-1}X_{\alpha,\alpha}(t-s-h)\phi(s)ds+\int_{0}^{t}(t-s)^{\alpha-1}X_{\alpha,\alpha}(t-s)Cu(s)ds
	\end{align*}
	where  
	\begin{align*}
		X_{\alpha}(t)=\mathscr{L}^{-1}\lbrace s^{\alpha-1}(s^{\alpha}I-A-Be^{-sh})^{-1}\rbrace(t),\quad and\quad X_{\alpha,\alpha}(t)=t^{1-\alpha}\int_{0}^{t}\frac{(t-s)^{\alpha-2}}{\Gamma(\alpha-1)}X_{\alpha}(s)ds .
	\end{align*}
\end{theorem}
\begin{theorem}
	Let $\alpha\in(0,1)$. If $A \in R^{n\times n}$ and $W\in R^{n}$.
	\begin{align*}
		\begin{cases}
			^{C}D^{\alpha}_{0+} x(t)=A x(t-h)+W,\quad t\in [0,T],\\
			x(t)=0,\quad -h \leq t\leq0.
		\end{cases}
	\end{align*}
	has a unique solution $x$ given the by formula
	\begin{align*}
		x(t)=\sum_{k=0}^{\infty}\frac{A^{k}(t-kh)^{\alpha(k+1)}}{\Gamma(\alpha(k+1)+1)}W,\quad t\in [0,T].
	\end{align*}
\end{theorem}
\begin{proof}
	
	Using Laplace transform, we obtain 
	
	\begin{align*}
		\mathscr{L}\lbrace^{C}D^{\alpha}_{0+} x(t)\rbrace(s)=&\mathscr{L}\bigg\lbrace  {^{RL}}I^{1-\alpha}_{0+} \bigg(\frac{dx(t)}{dt}\bigg)\bigg\rbrace(s)=s^{\alpha-1}\mathscr{L}\bigg\lbrace  \frac{dx(t)}{dt}\bigg\rbrace(s)\\
		=&s^{\alpha-1}\bigg(s\widehat{x}(s)-x(0)\bigg)=s^{\alpha}\widehat{x}(s),
	\end{align*}
	where $\widehat{x}(s)=\mathscr{L}\lbrace x\rbrace(s)$.
	\begin{align*}
		&\mathscr{L}\lbrace x(t-h)\rbrace(s)=\int_{0}^{\infty}e^{-st}x(t-h)dt\\
		&(If \quad we\quad make\quad the\quad substitution\quad t-h=\theta)\\
		=&e^{-sh}\int_{-h}^{\infty}e^{-s\theta}x(\theta)d\theta=e^{-sh}\int_{0}^{\infty}e^{-s\theta}x(\theta)d\theta=e^{-sh}\widehat{x}(s).
	\end{align*}
	\begin{align*}
		&\mathscr{L}\lbrace W\rbrace(s)=\int_{0}^{\infty}e^{-st}Wdt=s^{-1}W
	\end{align*}
	\begin{align*}
		&s^{\alpha}\widehat{x}(s)=Ae^{-sh}\widehat{x}(s)+s^{-1}W\\
		&\widehat{x}(s)=\frac{W}{s^{\alpha+1}-Ae^{-sh}s}
	\end{align*}
	by using inverse Laplace transform
	\begin{align*}
		x(t)=\mathscr{L}^{-1} \bigg\lbrace \frac{W}{s^{\alpha+1}-Ae^{-sh}s}\bigg\rbrace(t)=W \mathscr{L}^{-1} \bigg\lbrace \frac{1}{s^{\alpha+1}-Ae^{-sh}s}\bigg\rbrace(t)
	\end{align*}
	\begin{align*}
		&\frac{1}{s^{\alpha+1}-Ae^{-sh}s}=\frac{1}{s^{\alpha+1}}\frac{1}{1-\frac{Ae^{-sh}}{s^{\alpha}}}=\frac{1}{s^{\alpha+1}}\sum_{k=0}^{\infty}\frac{A^{k}e^{-skh}}{s^{\alpha k}}
		=\sum_{k=0}^{\infty}\frac{A^{k}e^{-skh}}{s^{\alpha(k+1)+1}}
	\end{align*}
	\begin{align*}
		\mathscr{L}^{-1} \bigg\lbrace \frac{1}{s^{\alpha+1}-Ae^{-sh}s}\bigg\rbrace(t)=\sum_{k=0}^{\infty} A^{k}\frac{(t-kh)^{\alpha(k+1)}}{\Gamma(\alpha(k+1)+1)}
	\end{align*}
	then 
	\begin{align*}
		x(t)=\sum_{k=0}^{\infty} A^{k}\frac{(t-kh)^{\alpha(k+1)}}{\Gamma(\alpha(k+1)+1)}W.
	\end{align*}
\end{proof}
From theorem 4.2, it follows that a solution of problem (4.5)-(4.6) is given by
\begin{align*}
	\lambda(t)=&\begin{bmatrix}
		\lambda_{1}(t)   \\[0.3em]
		\lambda_{2}(t)   \\[0.3em]
	\end{bmatrix}=\frac{t^{1/2}}{\Gamma(3/2)} \begin{bmatrix}
		-1   \\[0.3em]
		1   \\[0.3em]
	\end{bmatrix}+\begin{bmatrix}
		0 & 1   \\[0.3em]
		0 & 0   \\[0.3em]
	\end{bmatrix}\frac{(t-1/2)}{\Gamma(2)}\begin{bmatrix}
		-1   \\[0.3em]
		1   \\[0.3em]
	\end{bmatrix}\\
	=&\begin{bmatrix}
		\lambda_{1}(t)   \\[0.3em]
		\lambda_{2}(t)   \\[0.3em]
	\end{bmatrix}=\frac{t^{1/2}}{\Gamma(3/2)} \begin{bmatrix}
		-1   \\[0.3em]
		1   \\[0.3em]
	\end{bmatrix}+\frac{(t-1/2)}{\Gamma(2)}\begin{bmatrix}
		1  \\[0.3em]
		0   \\[0.3em]
	\end{bmatrix}\\
	=&\begin{bmatrix}
		\frac{(t-1/2)}{\Gamma(2)}-\frac{t^{1/2}}{\Gamma(3/2)}  \\[0.3em]
		\frac{t^{1/2}}{\Gamma(3/2)}   \\[0.3em]
	\end{bmatrix},\quad t\in [0,2].
\end{align*}
thus, 
\begin{align*}
	\begin{bmatrix}
		\lambda_{1}(t)   \\[0.3em]
		\lambda_{2}(t)   \\[0.3em]
	\end{bmatrix}=&\begin{bmatrix}
		\frac{(t-1/2)}{\Gamma(2)}-\frac{t^{1/2}}{\Gamma(3/2)}  \\[0.3em]
		\frac{t^{1/2}}{\Gamma(3/2)}   \\[0.3em]
	\end{bmatrix},\quad t\in [0,2].
\end{align*}
Consequently, condition (4.7) is equivalent to the following one
\begin{align*}
	u^{*}(t)-(t-1/2)u^{*}(t)=\min_{u\in[0,1]}\lbrace u-(t-1/2)u\rbrace=\begin{cases}
		\frac{3}{2}-t, \quad t\in[0,1],\\
		0,\quad\quad t\in[1,2].
	\end{cases}
\end{align*}
for $t\in [0,2]$ a.e.
Thus 
\begin{align*}
	u^{*}(t)=\begin{cases}
		1, \quad t\in [0,1] \quad a.e.\\
		0, \quad t\in [1,2] \quad a.e.
	\end{cases}
\end{align*}
From theorem 4.1, it follows that a solution of (4.1)-(4.2), corresponding to $u^{*}(\cdot)$ is given by
\begin{align*}
	y^{*}(t)=&X_{\alpha}(t)y(0)+A\int_{-h}^{0}(t-s)^{\alpha-1}X_{\alpha,\alpha}(t-s)y(s)ds+\int_{0}^{t}(t-s)^{\alpha-1}X_{\alpha,\alpha}(t-s)Bu^{*}(s)ds\\
	=&\begin{cases}
		\int_{0}^{t}(t-s)^{\alpha-1}X_{\alpha,\alpha}(t-s)Bds,\quad t\in [0,1]\quad a.e.\\
		\int_{0}^{1}(t-s)^{\alpha-1}X_{\alpha,\alpha}(t-s)Bds,\quad t\in [1,2]\quad a.e.
	\end{cases}
\end{align*}
For our problem, it is easy to check that, $X_{\alpha}(t)$ and $X_{\alpha,\alpha}(t-s)$ are given the following form
\begin{align*}
	X_{\alpha}(t)=\sum_{k=0}^{\infty}A^{k}\frac{(t-kh)^{\alpha k}}{\Gamma(\alpha k+1)},\quad and\quad
	X_{\alpha,\alpha}(t-s)=\sum_{k=0}^{\infty}A^{k}\frac{(t-s-kh)^{\alpha( k+1)}}{\Gamma(\alpha (k+1)+1)}
\end{align*}
then, we get
\begin{align*}
	y^{*}(t)=&\begin{cases}\begin{bmatrix}
			\frac{t^{1/2}}{\Gamma(2)}+\frac{t}{\Gamma(3/2)}-\frac{2t^{3/2}}{3} \\[0.3em]
			\frac{t^{1/2}}{\Gamma(2)}-\frac{t}{\Gamma(3/2)}-\frac{2t^{3/2}}{3}  \\[0.3em]
		\end{bmatrix}
		,\quad\quad\quad\quad\quad\quad\quad\quad t\in [0,1]\quad a.e.\\
		\begin{bmatrix}
			\frac{(t-1)^{1/2}-t^{1/2}}{\Gamma(2)}-\frac{1}{\Gamma(3/2)}+\frac{2((t-1)^{3/2}-t^{3/2})}{3} \\[0.3em]
			\frac{2((t-1)^{3/2}-t^{3/2})}{3}-\frac{(t-1)^{1/2}-t^{1/2}}{\Gamma(2)}-\frac{1}{\Gamma(3/2)}  \\[0.3em]
		\end{bmatrix},\quad t\in [1,2]\quad a.e.
	\end{cases}
\end{align*}
It means that the pair $(y^{*}(t),u^{*}(t))$ is the only pair, which can be a locally optimal solution of problem (4.1)-(4.4).
		\\

\end{document}